\documentclass{amsart}
\usepackage[english]{babel}
\usepackage{amssymb,graphicx,latexsym,pslatex}
\usepackage{mathtools}  
\mathtoolsset{showonlyrefs} 

\usepackage{hyperref}

\newcommand{\mathd}{\mathrm{d}}
\newcommand{\tmop}[1]{\ensuremath{\operatorname{#1}}}
\theoremstyle{plain}
\newtheorem{definition}{Definition}
\theoremstyle{plain}
\newtheorem{lemma}{Lemma}
\theoremstyle{remark}
\newtheorem{remark}{Remark}
\theoremstyle{plain}
\newtheorem{theorem}{Theorem}

\begin{document}

\title{Two quasi-local masses evaluated on surfaces with boundary}

\author{Xiaoxiang Chai}

\email{ChaiXiaoxiang@gmail.com}

\date{Nov. 15, 2018}

\begin{abstract}
  We study Hawking mass and the Huisken's isoperimetric mass evaluated on
  surfaces with boundary. The convergence to an ADM mass defined on
  asymptotically flat manifold with a non-compact boundary are proved.
\end{abstract}

{\maketitle}

\section{Introduction}

{\itshape{ADM mass}} defined in {\cite{arnowitt-canonical-1960}} has deep
connections with minimal surface theory and the geometry of scalar curvatures
as revealed by the seminal works {\cite{schoen1979,schoen-variational-1989}}.
Their theorems states that for an asymptotically flat manifold with
appropriate decay rate of the metric, if the manifold is of nonnegative scalar
curvature, then the ADM mass is nonnegative. This is extended to an
asymptotically flat manifold with a noncompact boundary by Almaraz, Barbosa
and de Lima {\cite{almaraz-positive-2016}}. First, we recall their definition
of an asymptotically flat manifold with a noncompact boundary,

\begin{definition}
  (Asymptotically flat with a noncompact boundary,
  {\cite{almaraz-positive-2016}}) We say that $(M, g)$ is asymptotically flat
  with decay rate $\tau > 0$ if there exists a compact subset $K \subset M$
  and a diffeomorphism $\Psi : M\backslash K \rightarrow \mathbb{R}_+^n
  \backslash \bar{B}_1^+ (0)$ such that the following asymptotics holds as $r
  \rightarrow + \infty$:
  \[ |g_{i j} (x) - \delta_{i j} | + r |g_{i j, k} | + r^2 |g_{i j, k l} | = o
     (r^{- \tau}) \]
  where $\tau > \frac{n - 2}{2}$.
\end{definition}

Here, $x = (x_1, \cdots, x_n)$ is the coordinate system induced by the
diffeomorphism $\Psi$, $r = |x|$, $g_{i j}$ are the components of $g$ with
respect to $x$, the comma denotes partial differentiation. We identify
$\mathbb{R}_+^n = \{x \in \mathbb{R}^n : x_1 \geqslant 0\}$ and $\bar{B}_1^+
(0) = \{x \in \mathbb{R}_+^n : |x| \leqslant 1\}$. In this work, we use the
Einstein summation convention with index ranges $i, j, k = 1, \cdots, n$ and
$a, b, c = 2, \cdots, n$. Observe that along $\partial M$, $\{\partial_a \}$
spans $T \partial M$ while $\partial_1$ points inwards of $M$. See the figure
below.

\begin{figure}[h]
  \resizebox{335pt}{205pt}{\includegraphics{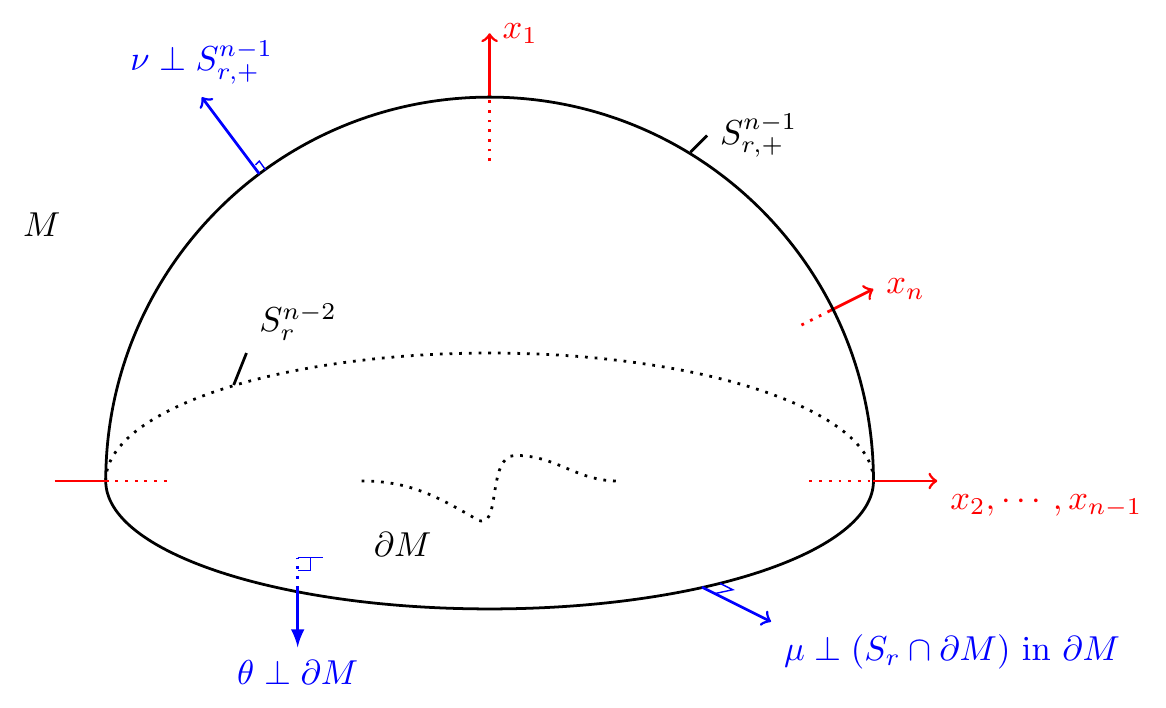}}
  \caption{A hemisphere $S_{r, +}^{n - 1}$ in an asymptotically
  flat manifold $M$ with a noncompact boundary $\partial M$.}
\end{figure}

Now we state the definition of the ADM mass for such asymptotically flat
manifolds.

\begin{definition}
  (ADM mass, {\cite{almaraz-positive-2016}}) The following quantity associated
  with $(M^n, g)$
  \begin{equation}
    m_{\ensuremath{\operatorname{ADM}}} := \lim_{r \rightarrow \infty} m_{}
    (r) \equiv \lim_{r \to + \infty} b_n \left\{ \int_{S_{r, +}^{n - 1}}
    (g_{ij, j} - g_{jj, i}) \nu^i \mathrm{d} \sigma + \int_{\partial S_{r,
    +}^{n - 1}} g_{a 1} \mu^a \mathrm{d} \theta \right\} \label{energy-def}
  \end{equation}
  defined for an asymptotically flat manifold $(M^n, g)$ with a non-compact
  boundary is called the ADM mass, $S_{r, +}^{n - 1} \subset M$ is a large
  standard Euclidean coordinate hemisphere of radius $r$ with outward unit
  normal $\nu$, and $\mu$ is the outward pointing unit co-normal to $\partial
  S_{r, +}^{n - 1}$ in $\partial M$. If the scalar curvature $R_g$ is
  integrable and mean curvature $H_g$ is integrable on $\partial M$, then
  $m_{\ensuremath{\operatorname{ADM}}}$ is well defined. Here $b_n =
  \frac{1}{2 (n - 1) \omega_{n - 1}}$ where $\omega_{n - 1}$ is the volume of
  $(n - 1)$-dimensional standard sphere.
\end{definition}

Almaraz, Barbosa and de Lima {\cite{almaraz-positive-2016}} established a
similar positive mass theorem.

\begin{theorem}
  Given an asymptotically flat manifold $(M, g)$ with $R_g, H_g \geqslant 0$,
  $R_g \in L^1 (M)$ and $H_g \in L^1 (\partial M)$. Then
  $m_{\ensuremath{\operatorname{ADM}}} \geqslant 0$. The mass
  $m_{\ensuremath{\operatorname{ADM}}}$ is zero if and only if $(M, g)$ is
  isometric to $(\mathbb{R}_+^n, \delta)$.
\end{theorem}

The ADM mass are defined for non-compact manifolds, it is then an interesting
question to consider quantities for bounded domains. This is the motivation
behind the notion of the {\itshape{quasi-local masses.}} We list a few
important quasi-local masses for interested readers: Brown-York mass
{\cite{brown-quasilocal-1993}}, Liu-Yau mass
{\cite{liu-positivity-2003,liu-positivity-2006}}, Wang-Yau mass
{\cite{wang-generalization-2007}}, Hawking mass
{\cite{hawking-gravitational-1968}} and Huisken's isoperimetric mass
{\cite{huisken-isoperimetric-2006}}. The criterion for a satisfying definition
of a quasi-local mass can be found in Liu and Yau's work
{\cite{liu-positivity-2006}}.

All these masses are evaluated on a closed 2-surface. The work
{\cite{almaraz-positive-2016}} leads us to consider the question:

{\itshape{Are there similar quantities evaluated on surfaces with boundary?}}

We will be concerned with extending the Hawking mass and Huisken's
isoperimetric mass. Our definition of the Hawking mass is the following:

\begin{definition}
  \label{def:hawking mass}(Hawking mass with boundary) Given a 2-surface
  $\Sigma \subset M^3$ with boundary $\partial \Sigma \neq \emptyset$
  intersecting $\partial M$ orthogonally, the Hawking mass is defined to be
  \begin{equation}
    m_H (\Sigma) = \left( \frac{| \Sigma |}{8 \pi} \right)^{\frac{1}{2}}
    \left( \chi (\Sigma) - \frac{1}{8 \pi} \int_{\Sigma} H^2 \right)
    \label{eq:hawking mass} .
  \end{equation}
  In particular, when $\Sigma$ is a disk i.e. $\chi (\Sigma) = 1$, then the
  mass takes the form
  \begin{equation}
    m_H (\Sigma) = \left( \frac{| \Sigma |}{8 \pi} \right)^{\frac{1}{2}}
    \left( 1 - \frac{1}{8 \pi} \int_{\Sigma} H^2 \right) . \label{eq:hawking
    disk}
  \end{equation}
\end{definition}

Observe that the form of the Hawking mass for a disk resembles the ususal
Hawking mass for a topological 2-sphere $\Sigma$ which writes as
\[ m_H (\Sigma) = (\frac{| \Sigma |}{16 \pi})^{\frac{1}{2}} \left( 1 -
   \frac{1}{16 \pi} \int_{\Sigma} H^2 \right) . \]
We also have the following isoperimetric mass with boundary.

\begin{definition}
  \label{isoperimetric boundary}(Isoperimetric mass with boundary) For any
  asymptotically flat manifold $(M^n, g)$ with a non-compact boundary and an
  open set $\Omega \subset M^3$ with finite perimeter and $\Omega \cap
  \partial M$ non-empty, the quantity
  \[ m_{\ensuremath{\operatorname{iso}}} (\Omega) = \frac{2}{\mathcal{H}^2
     (\partial^{\ast} \Omega \cap \ensuremath{\operatorname{int}}M)} [V
     (\Omega) - \frac{\sqrt{2} (\mathcal{H}^2 (\partial^{\ast} \Omega \cap
     \ensuremath{\operatorname{int}}M))^{3 / 2}}{6 \sqrt{\pi} V (\Omega)}] \]
  is called the isoperimetric quasi-local mass where $\mathcal{H}^2$ is the
  2-dimensional measure, $\partial^{\ast} \Omega$ is the reduced boundary of
  $\Omega$ and $V (\Omega)$ is the $3$-dimensional Hausdorff measure of
  $\Omega$. We can take $\Omega$ to be the standard Euclidean hemisphere, let
  $\mathcal{A} (r) =\mathcal{A} (S_{r, +}^2)$ be the area of the sphere $S_{r,
  +}^2$ and $V (r)$ is the volume of the region of $M$ bounded to the interior
  by $S_{r, +}^2$ and $\partial M$. The quantity
  $m_{\ensuremath{\operatorname{ISO}}}$ given by
  \begin{equation}
    m_{\ensuremath{\operatorname{ISO}}} = \limsup_{r \to \infty}
    \frac{2}{\mathcal{H}^2 (\partial^{\ast} \Omega \cap
    \ensuremath{\operatorname{int}}M)} (V (r) - \frac{\sqrt{2} (\mathcal{H}^2
    (\partial^{\ast} \Omega \cap \ensuremath{\operatorname{int}}M))^{3 / 2}}{6
    \sqrt{\pi} V (r)}) \label{huisken isoperimetric mass with boundary}
  \end{equation}
  is called isoperimetric mass with boundary.
\end{definition}

G. Huisken {\cite{huisken-isoperimetric-2006}} defined an isoperimetric mass
for the usual asymptotically flat manifold, see an expression in {\cite[p.
  51]{fan-large-sphere-2009}}. Volkman {\cite{volkmann-free-2014}} defined the same
isoperimetric mass with boundary which he called relative isoperimetric mass (up to a constant
multiple as ours), however, as we will show in the article, his definition is a special case.
Also, our proof of convergence to the ADM mass is more direct.

Fan, Shi and Tam {\cite{fan-large-sphere-2009}} proved that Huisken's
isoperimetric mass evaluated on standard coordinate spheres with increasing
radius approaches the ADM mass, which is the third item of Liu and Yau's
criterion {\cite[Introduction]{liu-positivity-2006}}. 

In this article, we will show that our boundary version of Hawking mass and
isoperimetric mass meet the third requirement of Liu and Yau's list of
criterion of a good quasi-local mass. The article is organized as follows:

In Section \ref{eval}, we evaluate the ADM mass {\eqref{energy-def}} via Ricci
tensor of $M$ and second fundamental form of $\partial M$. In Section
\ref{hawking derivation}, we show that the Hawking mass converges to the ADM
mass. In Section \ref{isoperimetric derivation}, we show that the
isoperimetric quasi-local mass converges to the ADM mass.

{\bfseries{Acknowledgements.}} This work is part of the author's PhD thesis at the Chinese
University of HK. He would like to thank sincerely his PhD advisor Prof. Martin Man-chun Li
for continuous encouragement and support.

\section{Evaluation via Ricci and second fundamental form}\label{eval}

{\bfseries{Notations}} We list the notations used in this note. We extend the
Euclidean distance $| \cdot |$ given by the diffeomorphism $\Psi$ to all of
$M$ by requring $| \cdot | : K \to [0, 1)$.
\begin{eqnarray}
  \Omega &  & \text{a domain intersecting } \partial M, \text{ and } \partial
  \Omega = \Sigma \cup \Pi, \nonumber\\
  &  & \Sigma^+ = \partial \Omega \cap \ensuremath{\operatorname{Int}}M,
  \nonumber\\
  &  & \Pi = \partial \Omega \cap \partial M, \nonumber\\
  g_{i j} &  & \text{metric on } g, \nonumber\\
  h_{a b} &  & \text{induced metric on } \partial M, \nonumber\\
  \nu &  & \text{normal of } \partial \Omega \text{ in } M, \nonumber\\
  \mu &  & \text{normal of } \partial M \text{ in } M, \nonumber\\
  \vartheta &  & \text{normal of } \partial \Sigma = \partial \Pi \text{ in }
  \Pi, \nonumber\\
  \mathrm{d} \sigma &  & (n - 1) \text{-dimensional volume element under
  metric } g, \nonumber\\
  \mathrm{d} \theta &  & (n - 2) \text{-dimensional volume element under
  metric } g, \nonumber\\
  X &  & X = (x_1, \cdots, x_n)  \text{ is the position vector under } \Psi,
  \nonumber\\
  D_r &  & \{x \in M \cup \partial M : |x| \leqslant r\}  \text{ for } r > 1,
  \nonumber\\
  \nabla &  & \text{connection on } M, \nonumber\\
  D &  & \text{induced connection on } \partial M, \nonumber\\
  A &  & \text{second fundamental form of } \partial M \text{ in } M,
  \nonumber\\
  H &  & \text{mean curvature of } \partial M \text{ in } M, \nonumber\\
  G &  & G =\ensuremath{\operatorname{Rc}}- \frac{1}{2} R g, \text{ the
  Einstein tensor of } M. \nonumber
\end{eqnarray}
The barred quantities are their Euclidean counterparts. $\Omega$ is very often
\[ D_r^+ = \{x \in M \cup \partial M : |x| \leqslant r\} \]
where $r > 1$, in this case we write
\[ \Sigma_r^+ = \{x \in \ensuremath{\operatorname{Int}}M : |x| = r\} \]
and
\[ \Pi_r = \{x \in \partial M : |x| \leqslant r\} . \]
Obviously, $\Sigma_r^+$ is a standard Euclidean coordinate hemisphere and
$\partial D_r^+ = \Sigma_r^+ \cup \Pi_r$. When there is no ambiguity, we drop
the subscript for simplicity.

We have similar expressions as {\cite[(1.3),(1.4) and
(1.5)]{miao-evaluation-2016}} where the usual asymptotically flat case is
handled.

\begin{theorem}
  \label{mass by tensors}Suppose that $(M, g)$ is an asymptotically flat
  manifold, then
  \begin{equation}
    m_{\ensuremath{\operatorname{ADM}}} = - c_n \lim_{r \rightarrow \infty}
    [\int_{\Sigma^+} G (X, \nu) \mathrm{d} \sigma + \int_{\partial \Sigma^+}
    (A - H g) (X, \vartheta) \mathrm{d} \theta], \label{ADM by tensors}
  \end{equation}
  here $c_n = \frac{1}{(n - 2) (n - 1) \omega_{n - 1}}$.
\end{theorem}

We collect some well known asymptotics in the following lemma:

\begin{lemma}
  We have the following asymptotics:
  \begin{align}
    2\ensuremath{\operatorname{Rc}}_{i j} & = g_{ki, kj} + g_{kj, ki} - g_{ij,
    kk} - g_{kk, ij} + O (r^{- 2 - 2 \tau}) \label{decay-ricci} ;\\
    R & = g_{i k, i k} - g_{k k, i i} + O (r^{- 2 - 2 \tau})
    \label{decay-scalar} ;\\
    & \\
    d \theta & = d \bar{\theta} + O (r^{- \tau}) d \bar{\theta} ;\\
    d \sigma & = d \bar{\sigma} + O (r^{- \tau}) d \bar{\sigma} ;\\
    & \\
    \nu - \bar{\nu} & = O (r^{- \tau}) ;\\
    \vartheta - \bar{\vartheta} & = O (r^{- \tau}) ;\\
    & \\
    \mu & = - g^{1 i} \partial_i / g^{11} = - \partial_1 + O (r^{- \tau}) ;\\
    & \\
    A_{a b} & = \frac{1}{2} (g_{a 1, b} + g_{b 1, a} - g_{a b, 1}) + O (r^{- 1
    - 2 \tau})\\
    & = O (r^{- 1 - \tau}) ;\\
    & \\
    H & = \frac{1}{2} (2 g_{1 a, a} - g_{a a, 1}) + O (r^{- 1 - 2 \tau})\\
    & = O (r^{- 1 - \tau}) .
  \end{align}
\end{lemma}

\begin{proof}
  The asymptotics of Ricci curvature and scalar curvature are well know, see
  for example {\cite[(2.2), (2.6)]{miao-evaluation-2016}}. Choose
  $\bar{g}$-orthonormal frame $e_i$, and let $e_1 = \bar{\nu}$. If $e_i =
  a_i^j \partial_j$, then we see that $a_i^j$ is an orthogonal matrix. So $g
  (e_i, e_j) = a_i^k a_j^l g (\partial_k, \partial_l) = \delta_{i j} + O (r^{-
  \tau})$ and this will give
  \begin{equation}
    \mathrm{d} \sigma = \mathrm{d} \bar{\sigma} + O (r^{- \tau}) \mathrm{d}
    \bar{\sigma} .
  \end{equation}
  Same reasoning applied on $\partial M$ will give for $\mathrm{d} \theta$ the
  asymptotics
  \begin{equation}
    \mathrm{d} \theta = \mathrm{d} \bar{\theta} + O (r^{- \tau}) \mathrm{d}
    \bar{\theta} .
  \end{equation}
  Let $\nu = \nu^i e_i$, then $1 = g (\nu, \nu) = \nu^i \nu^j g (e_i, e_j) =
  \nu^i \nu^j (\delta_{i j} + O (r^{- \tau}))$, we get $\sum_i (X^i)^2 = 1 + O
  (r^{- \tau})$. Similarly, from $g (\nu, e_a) = 0$, we get $0 = X^a + O (r^{-
  \tau})$, then
  \begin{equation}
    \nu - \bar{\nu} = \nu^a e_a + (\nu^1 - 1) \nu_1 = O (r^{- \tau}) .
  \end{equation}
  For $\vartheta$, we have as well $\vartheta - \bar{\vartheta} = (r^{-
  \tau})$.
  
  For the expression of $\mu$, $A$ and $H$ in in terms of the metric, refer to
  {\cite[(2.12) - (2.16)]{almaraz-positive-2016}}. For $\mu$,
  \begin{align}
    \mu & = - g^{1 i} \partial_i / (g^{1 k} g^{1 j} g_{j k})\\
    & = - g^{1 i} \partial_i / g^{11} = - \partial_1 + O (r^{- \tau}) .
  \end{align}
  
  On $\partial M$, we have
  \begin{align}
    A_{a b} & = - \langle \mu, \nabla_a \partial_b \rangle\\
    & = (g^{11})^{- 1 / 2} \Gamma_{a b}^1\\
    & = \frac{1}{2} (g_{a 1, b} + g_{b 1, a} - g_{a b, 1}) + O (r^{- 1 - 2
    \tau})\\
    & = O (r^{- 1 - \tau}) .
  \end{align}
  
  For $H : = H_{\partial M, M}$,
  \begin{align}
    H & = h^{a b} A_{a b}\\
    & = \frac{1}{2} (g^{11})^{1 / 2} (2 g_{1 a, a} - g_{a a, 1})\\
    & = \frac{1}{2} (2 g_{1 a, a} - g_{a a, 1}) + O (r^{- 1 - 2 \tau})\\
    & = O (r^{- 1 - \tau}) .
  \end{align}
  
  That concludes our proof of the asymptotics.
\end{proof}

\begin{proof}[Proof of Theorem \ref{mass by tensors}]We follow
{\cite{miao-evaluation-2016}}. Let
\begin{equation}
  I_1 = \int_{\Sigma^+ \cup \Pi} (- g_{ki, kj} - g_{kj, ki} + g_{ij, kk} +
  g_{kk, ij}) x^i \bar{\nu}^j \mathd \bar{\sigma}
\end{equation}
and
\begin{equation}
  I_2 = \int_{\Pi} (- g_{ki, kj} - g_{kj, ki} + g_{ij, kk} + g_{kk, ij}) x^i
  \bar{\nu}^j \mathd \bar{\sigma},
\end{equation}
then
\begin{align*}
  - 2 \int_{\Sigma^+} R_{i j} x^i \nu^j \mathd \sigma_{} & = \int_{\Sigma^+}
  (- g_{ki, kj} - g_{kj, ki} + g_{ij, kk} + g_{kk, ij}) x^i \bar{\nu}^j \mathd
  \bar{\sigma} + o (1)\\
  & = I_1 + I_2 + o (1) .
\end{align*}

To facilitate the computation, we can assume that manifold $M$ is
diffeomorphic to $\mathbb{R}_+^n$ and extend the metric smoothly to all of \
$\mathbb{R}_+^n$. Because we are evaluating at infinity, the result will be
independent of extensions. We compute $I_1$ first,
\begin{align*}
  I_1 & = \int_{\Sigma^+ \cup \Pi} (- g_{ki, kj} - g_{kj, ki} + g_{ij, kk} +
  g_{kk, ij}) x^i \bar{\nu}^j \mathd \bar{\sigma}\\
  & = (n - 2) \int_{\Sigma^+ \cup \Pi} (g_{k j, k} - g_{k k, j}) \bar{\nu}^j
  \mathd \bar{\sigma} + \int_{\Sigma^+ \cup \Pi} (- g_{k j, k j} + g_{k k, j
  j}) x^i \bar{\nu}^i \mathd \bar{\sigma}\\
  & = (n - 2) \int_{\Sigma^+} (g_{k j, k} - g_{k k, j}) \bar{\nu}^j \mathd
  \bar{\sigma} + \int_{\Sigma^+} (- g_{k j, k j} + g_{k k, j j}) x^i
  \bar{\nu}^i \mathd \bar{\sigma}\\
  & \hspace{10.6em} - (n - 2) \int_{\Pi} (g_{a 1, a} - g_{a a, 1}) \mathd
  \bar{\sigma} .
\end{align*}
In the above the second equality is just {\cite[formula
(2.4)]{miao-evaluation-2016}}. This identity is easily proved using
integration by parts for $C^3$ metrics; for $C^2$ metrics,
{\cite{miao-evaluation-2016}} used approximation. Note that on $\Pi$, all
components of $\bar{\nu}$ is zero except that $\bar{\nu}^1 = - 1$ and $x^i
\bar{\nu}^i = 0$.

We compute $I_2$ using $\bar{\nu}^1 = - 1$ again,
\begin{align*}
  & - I_2\\
  = & \int_{\Pi} (g_{ki, kj} + g_{kj, ki} - g_{ij, kk} - g_{kk, ij}) x^i
  \bar{\nu}^j \mathd \bar{\sigma}\\
  = & \int_{\Pi} (g_{b a, b 1} + g_{b 1, b a} - g_{a 1, b b} - g_{b b, a 1})
  x^a \mathd \bar{\sigma}\\
  = & \int_{\partial \Pi} g_{b a, 1} x^a \bar{\vartheta}^b \mathd \bar{\theta}
  - \int_{\Pi} g_{b a, 1} \frac{\partial x^a}{\partial x^b} \mathd
  \bar{\sigma}\\
  & \hspace{2.0em} + \int_{\partial \Pi} g_{b 1, a} x^a \bar{\vartheta}^b
  \mathd \bar{\theta} - \int_{\Pi} g_{b 1, a} \frac{\partial x^a}{\partial
  x^b} \mathd \bar{\sigma}\\
  & \hspace{2.0em} - \int_{\partial \Pi} g_{a 1, b} x^a \bar{\vartheta}^b
  \mathd \bar{\theta} + \int_{\Pi} g_{a 1, b} \frac{\partial x^a}{\partial
  x^b} \mathd \bar{\sigma}\\
  & \hspace{2.0em} - \int_{\partial \Pi} g_{b b, 1} x^a \bar{\vartheta}^a
  \mathd \bar{\theta} + \int_{\Pi} g_{b b, 1} \frac{\partial x^a}{\partial
  x^a} \mathd \bar{\sigma}\\
  = & \int_{\partial \Pi} g_{b a, 1} x^a \bar{\vartheta}^b \mathd
  \bar{\theta}\\
  & \hspace{2.0em} - \int_{\partial \Pi} g_{b b, 1} x^a \bar{\vartheta}^a
  \mathd \bar{\theta} + (n - 2) \int_{\Pi} g_{b b, 1} \mathd \bar{\sigma} .
\end{align*}

For the boundary term,
\begin{align*}
  & 2 \int_{\partial \Sigma} (A_{a b} - H g_{a b}) x^a \vartheta^b \mathd
  \theta\\
  = & \int_{\partial \Pi} (2 g_{1 a, b} - g_{a b, 1}) x^a \bar{\vartheta}^b
  \mathd \bar{\theta} - \int_{\partial \Pi} (2 g_{1 a, a} - g_{a a, 1}) x^b
  \bar{\vartheta}^b \mathd \bar{\theta} + o (1)\\
  = & - \int_{\partial \Pi} g_{a b, 1} x^a \bar{\vartheta}^b \mathd
  \bar{\theta} + \int_{\partial \Pi} g_{b b, 1} x^a \bar{\vartheta}^a \mathd
  \bar{\theta}\\
  & \hspace{2.0em} + 2 \int_{\partial \Pi} g_{1 a, b} x^a \bar{\vartheta}^b
  \mathd \bar{\theta} - 2 \int_{\Pi} (g_{1 a, a} x^b)_{, b} \mathd
  \bar{\sigma} + o (1)\\
  = & - \int_{\partial \Pi} g_{a b, 1} x^a \bar{\vartheta}^b \mathd
  \bar{\theta} + \int_{\partial \Pi} g_{b b, 1} x^a \bar{\vartheta}^a \mathd
  \bar{\theta}\\
  & \hspace{2.0em} + 2 \int_{\partial \Pi} g_{1 a, b} x^a \bar{\vartheta}^b
  \mathd \bar{\theta} - 2 \int_{\Pi} (g_{1 a, a b} x^b + g_{1 a, a}
  \frac{\partial x^b}{\partial x^b}) \mathd \bar{\sigma} + o (1)\\
  = & - \int_{\partial \Pi} g_{a b, 1} x^a \bar{\vartheta}^b \mathd
  \bar{\theta} + \int_{\partial \Pi} g_{b b, 1} x^a \bar{\vartheta}^a \mathd
  \bar{\theta} + 2 \int_{\partial \Pi} g_{1 a, b} x^a \bar{\vartheta}^b \mathd
  \bar{\theta}\\
  & \hspace{2.0em} - 2 (n - 1) \int_{\Pi} g_{1 a, a} \mathd \bar{\sigma}\\
  & \hspace{2.0em} - 2 \int_{\partial \Pi} g_{1 a, b} x^b \bar{\vartheta}^a
  \mathd \bar{\theta} + 2 \int_{\Pi} g_{1 a, b} \frac{\partial x^b}{\partial
  x^a} \mathd \bar{\sigma} + o (1)\\
  = & - \int_{\partial \Pi} g_{a b, 1} x^a \bar{\vartheta}^b \mathd
  \bar{\theta} + \int_{\partial \Pi} g_{b b, 1} x^a \bar{\vartheta}^a \mathd
  \bar{\theta} - 2 (n - 2) \int_{\Pi} g_{1 a, a} \mathd \bar{\sigma} + o (1) .
\end{align*}

Scalar curvature has asymptotics
\begin{equation}
  R = g_{i k, i k} - g_{k k, i i} + o (r^{- 2 - 2 \tau}),
\end{equation}
hence
\begin{align*}
  & \int_{\Sigma^+} (- g_{k j, k j} + g_{k k, j j}) x^i \bar{\nu}^i \mathd
  \bar{\sigma}\\
  = & - \int_{\Sigma^+} R x^i \bar{\nu}^i \mathd \bar{\sigma} + o (1)\\
  = & - \int_{\Sigma^+} R g (X, \nu) \mathd \sigma + o (1) .
\end{align*}

Collect all calculations of $I_1$, $I_2$ and $2 \int_{\partial \Sigma^+} (A_{a
b} - H g) x^a \vartheta^b \mathd \theta$, we have
\begin{align*}
  & \int_{\Sigma^+} (- 2 \tmop{Rc} + R g) (X, \nu) \mathd \sigma - 2
  \int_{\partial \Sigma^+} (A - H g) (X, \vartheta) \mathd \theta\\
  = & (n - 2) [\int_{\Sigma^+} (g_{k j, k} - g_{k k, j}) \bar{\nu}^j \mathd
  \bar{\sigma} + \int_{\Pi} g_{a 1, a} \mathd \bar{\sigma}] + o (1)\\
  = & (n - 2) [\int_{\Sigma^+} (g_{k j, k} - g_{k k, j}) \bar{\nu}^j \mathd
  \bar{\sigma} + \int_{\partial \Sigma^+} g_{a 1} \bar{\vartheta}^a \mathd
  \bar{\theta}] + o (1)
\end{align*}
and the theorem is proved.
\end{proof}

\begin{remark}
  The proof of Theorem \ref{mass by tensors} works for slightly more general
  bounded open sets. See similar statements of {\cite[Theorem
  2.1]{miao-evaluation-2016}}.
\end{remark}

\subsection{Another Proof}

Here we provide an alternate proof of {\eqref{ADM by tensors}}. The proof of
the standard case is due to Herzlich {\cite{herzlich-computing-2016}}, we
extend the proof to our settings. We have the following identity.

\begin{lemma}
  \label{lm:integrated Bianchi}(Integrated Bianchi identity) Given a
  Riemannian manifold $(M^n, g)$, we use the same notation with the ones given
  at the start of this section. $X$ is a conformal Killing vector field i.e.
  \[ \nabla^i X^j = \frac{1}{n} g^{i j} \ensuremath{\operatorname{div}}X. \]
  $\Omega \subset M$ is a bounded open set whose $\partial \Omega$ decomposes
  as the union of $\Sigma^+ = \partial \Omega \cap
  \ensuremath{\operatorname{int}}M$ and $\Pi = \partial \Omega \cap \partial
  M$ who share a common boundary $\partial \Sigma^+ = \partial \Pi$, $X$ is
  tangent to $\partial M$ and $X$ is also conformal Killling on $(\Pi, h)$
  i.e.
  \[ D^a X^b = \frac{1}{n - 1} h^{a b} \ensuremath{\operatorname{div}}_{\Pi}
     X. \]
  We have then
  \begin{align}
    & \int_{\partial \Pi} (A - H g) (X, \vartheta) + \int_{\Sigma^+} G (X,
    \nu)\\
    = & - \frac{n - 2}{2 n} \int_{\Omega} R\ensuremath{\operatorname{div}}X -
    \frac{n - 2}{n - 1} \int_{\Pi} H\ensuremath{\operatorname{div}}_{\Pi} X.
    \label{integrated Bianchi identity}
  \end{align}
\end{lemma}

\begin{proof}
  By divergence theorem,
  \begin{align*}
    \int_{\partial \Omega} G (X, \nu) & = \int_{\Omega} \nabla^i (G_{i j}
    X^j)\\
    & = \int_{\Omega} X^j \nabla^i G_{i j} + G_{i j} \nabla^i X^j\\
    & = \int_{\Omega} G_{i j} (\frac{1}{n} \ensuremath{\operatorname{div}}X
    g^{i j})\\
    & = \frac{1}{n} \int_{\Omega} \ensuremath{\operatorname{div}}X G_{i j}
    g^{i j}\\
    & = \frac{2 - n}{2 n} \int_{\Omega} R\ensuremath{\operatorname{div}}X.
  \end{align*}
  
  Since $X$ is tangent to $\partial M$, $\nu$ is normal to $\partial M$ along
  $\Pi$, we can use the Gauss-Codazzi equation
  \[ G (X, \nu) = X^b D^a (A_{a b} - H h_{a b}) \]
  to deal with $G (X, \nu)$ integrated on $\Pi$, we have
  \begin{align*}
    \int_{\Pi} G (X, \nu) & = \int_{\Pi} X^b D^a (A_{a b} - H h_{a b})\\
    & = \int_{\partial \Pi} (A - H g) (X, \vartheta) - \int_{\Pi} (A_{a b} -
    H h_{a b}) D^a X^b\\
    & = \int_{\partial \Pi} (A - H g) (X, \vartheta) - \frac{1}{n - 1}
    \int_{\Pi} (A_{a b} - H h_{a b}) \ensuremath{\operatorname{div}}_{\Pi} X
    h^{a b}\\
    & = \int_{\partial \Pi} (A - H g) (X, \vartheta) + \frac{n - 2}{n - 1}
    \int_{\Pi} H\ensuremath{\operatorname{div}}_{\Pi} X.
  \end{align*}
  
  Combining the two formulas above, we obtain
  \[ - \frac{n - 2}{n - 1} \int_{\Pi} H\ensuremath{\operatorname{div}}_{\Pi} X
     - \frac{n - 2}{2 n} \int_{\Omega} R\ensuremath{\operatorname{div}}X =
     \int_{\Sigma^+} G (X, \nu) + \int_{\partial \Sigma^+} (A - H g) (X,
     \vartheta) . \]
  
\end{proof}

\begin{proof}[Another proof of Theorem \ref{ADM by tensors}]For any
sufficiently large $R > 1$, we define a cutoff function $\chi_R$ such that it
is zero inside the hemisphere of radius $R / 2$, equals 1 outside the
hemisphere of radius $\frac{3}{4} R$ and it satisfies the following bounds on
its derivatives
\[ | \nabla \chi_R | \leqslant C R^{- 1}, \hspace{1em} | \nabla^2 \chi_R |
   \leqslant C R^{- 2}, \hspace{1em} | \nabla^3 \chi_R | \leqslant C R^{- 3}
\]
for some universal constant $C$ not depending on $R$. We write in short $\chi
= \chi_R$ when there is no confusion. We then define a half-annulus $A_R =
B_R^+ \setminus B_{R / 4}^+$ and a metric $h = g \chi + (1 - \chi) \delta$,
here $\delta$ is the standard Euclidean metric. Now we use subscript or
supscript $e, h, g$ on a term to denote that this term is evaluated under
corresponding metric.

Using {\eqref{integrated Bianchi identity}} by assigning $\Omega = A_R$,
\begin{align*}
  & \int_{\partial \Pi_R} (A - H g) (X, \vartheta) + \int_{\Sigma_R} G (X,
  \nu)\\
  = & \int_{\Omega_R} G_{i j} \nabla^i X^j + \int_{\Pi_R} (A_{a b} - H h_{a
  b}) D^a X^b\\
  = & \frac{2 - n}{2 n} \int_{\Omega_R} R \tmop{div} X + \int_{\Omega_R} G_{i
  j} (\nabla^i X^j)^o\\
  & \hspace{2.0em} + \frac{2 - n}{n - 1} \int_{\Pi_R} H \tmop{div}_{\Pi} X +
  \int_{\Pi_R} (A_{a b} - H h_{a b}) (D^a X^b)^o\\
  = : & I_1 + I_2 + J_1 + J_2,
\end{align*}
where $(\nabla^i X^j)^{\omicron} = \nabla^i X^j - \frac{1}{n} g^{i j}
\tmop{div} X$ and $(D^a X^b)^o = D^a X^b - \frac{1}{n - 1} h^{a b}
\tmop{div}_{\Pi} X$.

We estimate these terms separately. For $I_1$,
\begin{align*}
  & \int_{\Omega_R} R \tmop{div} X \mathd \mathcal{H}^n_g - \int_{\Omega_R} R
  \tmop{div}^e X \mathd \mathcal{H}^n_e\\
  \leqslant & C \sup (| \Gamma | |X| |R|) \mathcal{H}^n_g (\Omega_R)\\
  \leqslant & C R^{- \tau - 1} R R^{- 2 - \tau} R^n\\
  \leqslant & C R^{n - 2 \tau - 2} = o (1) ;
\end{align*}

For $J_1$,
\begin{align*}
  & \int_{\Pi_R} H \tmop{div}_{\Pi} X d\mathcal{H}^{n - 1}_g - \int_{\Pi_R} H
  \tmop{div}_{\Pi}^e X \mathd \mathcal{H}^{n - 1}_e\\
  = & \int_{S_R} H (\delta^h) X - \int_{S_R} H (\delta^e) X\\
  \leqslant & C \sup (| \Gamma | |X| |H|) \mathcal{H}^{n - 1}_e (S_R)\\
  \leqslant & C R^{- 1 - \tau} R R^{- 1 - \tau} R^{n - 1}\\
  \leqslant & C R^{n - 2 \tau - 2} = o (1) .
\end{align*}
It is easy to see that the same reasoning applies to the terms $I_2$ and
$J_2$, we get $I_2, J_2 = o (1)$. We have that
\begin{align*}
  & \int_{\partial \Pi_R} (A - H g) (X, \vartheta) + \int_{\Sigma_R^+} G (X,
  \nu)\\
  = & \frac{2 - n}{2 n} \int_{\Omega_R} R \tmop{div}^e X \mathd
  \mathcal{H}^n_e + \frac{2 - n}{n - 1} \int_{S_R} H \tmop{div}^e_{\Pi} X
  \mathd \mathcal{H}^{n - 1}_e + o (1)\\
  = & K_1 + K_2 + o (1) .
\end{align*}

Now we estimate $K_1$ and $K_2$.

\begin{align*}
  K_1 & = \int_{\Omega_R} R \tmop{div}_e X \mathd \mathcal{H}^n_e\\
  & = \int_{\Omega_R} \tmop{div} X Q (e, g) \mathd \mathcal{H}^n_e + n
  \int_{\Sigma_R^+} \langle \mathbb{U}, n \rangle \mathd \mathcal{H}^n_e\\
  & \hspace{2.0em} - n \int_{\Sigma_{R / 4}^+} \langle \mathbb{U}, n \rangle
  \mathd \mathcal{H}^n_e + n \int_{\Pi_R \sim \Pi_{R / 4}} \langle \mathbb{U},
  n \rangle \mathd \mathcal{H}^n_e\\
  & = n \int_{\Sigma_R^+} \langle \mathbb{U}, n \rangle \mathd
  \mathcal{H}^n_e + n \int_{\Pi_R \sim \Pi_{R / 4}} \langle \mathbb{U}, n
  \rangle \mathd \mathcal{H}^n_e + o (1) .
\end{align*}

And for $K_2$,
\begin{align*}
  2 K_2 & = 2 \int_{S_R} H \tmop{div}^e_{\Pi} X \mathd \mathcal{H}^n_e\\
  & = (n - 1) \int_{S_R} 2 H \mathd \mathcal{H}^n_e + o (1)\\
  & = (n - 1) \int_{S_R} (2 v_{1 a, a} - v_{a a, 1}) \mathd \mathcal{H}^n_e +
  o (1) .
\end{align*}

Along $S_R$, since $n = - \partial_1$, we have that
\[ \langle \mathbb{U}, n \rangle = v_{i j, j} n^i - v_{j j, i} n^i = - v_{1 j,
   j} + v_{j j, 1} = - v_{1 a, a} + v_{a a, 1} . \]
So
\begin{align*}
  & \int_{\partial \Pi_R} (A - H g) (X, \vartheta) + \int_{\Sigma_R^+} G (X,
  \nu)\\
  = & \frac{2 - n}{2} \int_{\Sigma_R^+} \langle \mathbb{U}, n \rangle \mathd
  \mathcal{H}^n_e + \frac{2 - n}{2} \int_{S_R} (- v_{1 a, a} + v_{a a, 1})
  \mathd \mathcal{H}^n_e\\
  & \hspace{2.0em} \frac{2 - n}{2} \int_{S_R} (2 v_{1 a, a} - v_{a a, 1})
  \mathd \mathcal{H}^n_e + o (1)\\
  = & \frac{2 - n}{2} [\int_{\Sigma_R^+} \langle \mathbb{U}, n \rangle \mathd
  \mathcal{H}^n_e + \int_{S_R} g_{1 a, a} \mathd \mathcal{H}^n_e] + o (1)\\
  = & \frac{2 - n}{2} [\int_{\Sigma_R^+} \langle \mathbb{U}, n \rangle \mathd
  \mathcal{H}^n_e + \int_{\partial \Sigma_R^+} g_{1 a} \vartheta^a \mathd
  \mathcal{H}^{n - 2}_e] + o (1)\\
  = & (2 - n) (n - 1) \omega_{n - 1} m_{\tmop{ADM}} + o (1)
\end{align*}
thus proving {\eqref{ADM by tensors}}.
\end{proof}

\subsection{A graphical example}

Motivated by {\cite{lam-graphs-2010}}, we give a graphical example where
$\partial M$ is given as a graph of a function. Set $X = (x_1, \ldots, x_n)$
and $B_{\rho}$ be the ball centered at $0$ with radius $\rho$. Let $n
\geqslant 2$, given a function $u : \mathbb{R}^n \setminus B_1 \to
\mathbb{R}$, the set
\[ \{(x_0, X) : X \in \mathbb{R}^n \setminus B_1, x_0 \geqslant u (X)\} \]
is our $M \setminus K$. We can extend $M \setminus K$ arbitrarily. Suppose
that $u$ is bounded and has asymptotics
\begin{equation}
  |x|^{n - 1} |D u| + |x|^n |D^2 u| = O (1) . \label{asymptotics of u}
\end{equation}
$\partial M$ contains the set $\{(u (X), X) : X \in \mathbb{R}^n \setminus B_1
\}$, therefore $M$ has a non-compact boundary. The asymptotic flat structure
of $M$ can be given by the map
\[ \Psi : (x_0, X) \mapsto (x_0 + v (x_0) u (X), X) \]
where $v (0) = 1$ with $v$ has the same asymptotics of $u$ and $M$ carries
metric $g$ induced from this map from $\mathbb{R}^{n + 1}$. Let $e_i$ where $i
= 0, 1, \cdots, n$ be the standard orthonormal basis of Euclidean space
$\mathbb{R}^{n + 1}$ and $\langle \cdot, \cdot \rangle$ be the standard
Euclidean inner product. Let $\hat{e}_i$ be the vector field induced by the
map $\Psi$ and the coordinate system $(x_0, X)$. It is easy to see that
\begin{align}
  \hat{e}_0 & = (1 + u (X) \partial_0 v) e_0,\\
  \hat{e}_a & = v (x_0) \partial_a u e_0 + e_a,
\end{align}
and the metric takes the form
\begin{align}
  g_{0 0} = \langle \hat{e}_0, \hat{e}_0 \rangle = & 1 + 2 u (X) \partial_0 v
  + u (X)^2 | \partial_0 v|^2,\\
  g_{0 a} = \langle \hat{e}_0, \hat{e}_0 \rangle = & v (x_0) \partial_a u + u
  (X) v (x_0) \partial_0 v \partial_a u,\\
  g_{a a} = \langle \hat{e}_a, \hat{e}_0 \rangle = & v (x_0)^2 | \partial_a
  u|^2 + 1.
\end{align}
Here the index $a$ ranges $1, \ldots, n$.

Since when $x_0^2 + |X|^2 = r$ under standard Euclidean metric, either $|x_0
| \geqslant r / 2$ or $|X| \geqslant r / 2$, considering also that $u, v$ are
bounded, $|g_{i j} - \delta_{i j} | = O (r^{1 - n})$. The decay of first and
second order derivatives of the metric follows similarly. So $(M^{n + 1}, g)$
is asymptotically flat with a non-compact boundary. The asymptotics specified
by {\eqref{asymptotics of u}} is related to complete minimal surfaces regular
at infinity, see the work of Schoen {\cite{schoen-uniqueness-1983}} and
Volkman {\cite{volkmann-free-2014}}.

Calculating the ADM mass using metric induced from this map is complicated.
But we are able to derive a simple formula for the mass in this graphical
case.

\begin{theorem}
  If M is an asymptotically flat manifold with a non-compact boundary,
  $\partial M$ is given by a function $u$ on $\mathbb{R}^n \setminus B_1$ with
  asymptotics {\eqref{asymptotics of u}}, then the mass has the simpler form
  \[ m_{\ensuremath{\operatorname{ADM}}} = (n - 1) \lim_{\rho \to \infty}
     \int_{\partial B_{\rho}} \partial_{\rho} u \mathrm{d} \bar{\theta} . \]
  where $\partial_{\rho}$ is radial unit normal to $\partial B_{\rho}$ in
  $B_{\rho}$ with respect to Euclidean metric and $\mathrm{d} \bar{\theta}$ is
  the Euclidean $(n - 1)$-dimensional volume element.
\end{theorem}

\begin{proof}
  Without loss of generality, we extend $u$ to all of $\mathbb{R}^n$ i.e. $u :
  \mathbb{R}^n \to \mathbb{R}$. We identify $B_{\rho}$ as the set $\{(0, X) :
  |X| \leqslant \rho\} \subset \mathbb{R}^{n + 1}$, then $\partial B_{\rho} =
  \{(0, X) : |X| = \rho\}$.
  
  We have by well known formulas (see for example {\cite[Section
  2]{eichmair-plateau-2009}})
  \begin{align*}
    v & = \sqrt{1 + |D u|^2} ;\\
    A_{a b} & = v^{- 1} u_{a b} = O (\rho^{- n}) ;\\
    g_{a b} & = \delta_{a b} + u_a u_b\\
    & = \delta_{a b} + O (\rho^{2 - 2 n}) ;\\
    g^{a b} & = \delta^{a b} - v^{- 2} u^a u^b ;\\
    H & = v^{- 1} u_{a b} (\delta^{a b} - v^{- 2} u^a u^b)\\
    & = v^{- 1} \Delta u - v^{- 3} u_{a b} u^a u^b\\
    & = v^{- 1} \Delta u + O (\rho^{- 3 n + 2})
  \end{align*}
  where $g$ gives the induced metric on the graph. Now we do calculations on
  $(B_{\rho}, g)$. The normal $\vartheta$ to $\partial B_{\rho}$ in $B_{\rho}$
  is given by
  \[ \vartheta^b = \hat{g}^{b c} \vartheta_c = \hat{g}^{b c}
     \frac{x_c}{|X|_{\hat{g}}} = \partial_{\rho} + O (\rho^{2 - 2 n}) . \]
  So $X = \rho \partial_{\rho} = \rho \vartheta + O (\rho^{3 - 2 n})$ and
  \begin{align*}
    & (A - H g) (X, \vartheta)\\
    = & \rho A (\partial_{\rho}, \partial_{\rho}) - \rho H\\
    = & \rho v^{- 1} \ensuremath{\operatorname{Hess}}u (\partial_{\rho},
    \partial_{\rho}) - \rho v^{- 1} \Delta u + O (\rho^{- 3 n + 3})\\
    = & \rho \ensuremath{\operatorname{Hess}}u (\partial_{\rho},
    \partial_{\rho}) - \rho \Delta u + O (\rho^{- 3 n + 3}) .
  \end{align*}
  On $\partial B_{\rho}$, we decompose the standard Laplacian of
  $\mathbb{R}^n$
  \[ \Delta u = \Delta_{\partial B_{\rho}} u
     +\ensuremath{\operatorname{Hess}}u (\partial_{\rho}, \partial_{\rho}) +
     H_{\partial B_{\rho}, B_{\rho}} \partial_{\rho} u \]
  where $H_{\partial B_{\rho}, \mathbb{R}^n} = (n - 1) / \rho$ is the mean
  curvature of a sphere of radius $\rho$ in $\mathbb{R}^n$.
  
  Hence
  \[ (A - H g) (X, \vartheta) = - (n - 1) \partial_{\rho} u - \rho
     \Delta_{\partial B_{\rho}} u + O (\rho^{- 3 n + 3}) . \]
  Suppose that $\Omega$ is an region intersecting $M$ at $\partial B_{\rho}$
  satisfying requirements of {\cite[Theorem 2.1]{miao-evaluation-2016}}.
  Because $M$ is an unbounded region in $\mathbb{R}^n$, now the mass by
  Theorem \ref{mass by tensors} has only a boundary term, then
  \begin{align*}
    m_{\ensuremath{\operatorname{ADM}}} = & - \int_{\partial B_{\rho}} (A - H
    g) (X, \vartheta) \mathrm{d} \theta + o (1)\\
    = & \int_{\partial B_{\rho}} [(n - 1) \partial_{\rho} u + \rho
    \Delta_{\partial B_{\rho}} u] \mathrm{d} \bar{\theta} + o (1)\\
    = & (n - 1) \int_{\partial B_{\rho}} \partial_{\rho} u \mathrm{d}
    \bar{\theta} + o (1) .
  \end{align*}
  
  The calculation obviously works for nearly round surfaces (see the
  definition in {\cite[Definition 2.1]{miao-quasi-local-2017}}) , we omit the
  details.
\end{proof}

\begin{remark}
  We point out that this expression is the same as {\cite[Definition
  2.5]{volkmann-free-2014}} and we refer readers to the positive mass theorem
  proved in there.
\end{remark}

\section{Hawking mass derivation}\label{hawking derivation}

We are going to extend Theorem 1.2 in {\cite{miao-quasi-local-2017}} to the
boundary case and thus obtain a formula of Hawking type mass. We do not pursue
generalities to deal with nearly round surfaces here and we choose coordinate
hemispheres as approximating surfaces. We have the following two trivial
lemmas.

\begin{lemma}
  Let $X = (x^1, \cdots, x^n)$ be the position vector given by coordinates
  near infinity. For each large $\rho$, on $\Sigma_{\rho}$
  \begin{equation}
    |X - \rho \bar{\nu} | = O (\rho^{1 - \tau}) \hspace{1em}
  \end{equation}
  and
  \begin{equation}
    |X - \rho \bar{\vartheta} | = O (\rho^{1 - \tau}) \text{ along } \partial
    \Sigma_{\rho} .
  \end{equation}
\end{lemma}

\begin{lemma}
  For large $\rho$, the volume $| \Sigma_{\rho} |$ satisfies
  \begin{equation}
    \left( \frac{2| \Sigma_{\rho} |}{\omega_{n - 1}} \right)^{\frac{1}{n - 1}}
    = \rho (1 + O (\rho^{- \tau})) .
  \end{equation}
\end{lemma}

\begin{theorem}
  \label{hawking type mass}(Hawking type mass) Let $\{\Sigma_{\rho} \}$ be a
  family of coordinate hemispheres with radius $\rho > 1$ in an asymptotically
  flat manifold $(M, g)$ of dimension $n \geqslant 3$. Let $\mu'$ be normal to
  $\partial \Sigma$ in $\Sigma$. Then
  \begin{equation}
    c_n \left( \frac{2| \Sigma |}{\omega_{n - 1}} \right)^{\frac{1}{n - 1}}
    \left[ \int_{\Sigma} (S - \frac{n - 2}{n - 1} H_{\Sigma, M}^2) \mathrm{d}
    \sigma + 2 \int_{\partial \Sigma} H_{\partial \Sigma, \Sigma} + \langle
    \vartheta, \mu' \rangle H_{\partial \Sigma, \partial M} \mathrm{d} \theta
    \right] \label{eq:hawking}
  \end{equation}
  converges to the ADM mass $m_{\ensuremath{\operatorname{ADM}}}$ as $\rho \to
  \infty$ where $S$ is the scalar curvature of $\Sigma$.
\end{theorem}

\begin{proof}
  By Gauss equation,
  \begin{align*}
    G (X, \nu) & = G (\rho \bar{\nu}, \nu)\\
    & = \rho G (\nu, \nu) + O (\rho^{- 1 - 2 \tau})\\
    & = \frac{1}{2} \rho (H^2_{\Sigma, M} - |A|^2 - S) + O (\rho^{- 1 - 2
    \tau})
  \end{align*}
  
  where $S$ is the scalar curvature of $\Sigma$.
  
  On $\partial \Pi$,
  \begin{eqnarray*}
    (A - H g) (X, \vartheta) & = & (A - H g) (\rho \bar{\vartheta},
    \vartheta)\\
    & = & \rho (A - H g) (\vartheta, \vartheta) + O (\rho^{- 1 - 2 \tau})\\
    & = & \rho A (\vartheta, \vartheta) - \rho H + O (\rho^{- 1 - 2 \tau})\\
    & = & - \sum_i \langle \nabla_{e_i} \mu, e_i \rangle + O (\rho^{- 1 - 2
    \tau}) .
  \end{eqnarray*}
  where $e_i$'s are orthonormal basis of $T \partial \Sigma$. Let $\mu'$ be
  the normal to $\partial \Sigma$ in $\Sigma$, moreover,
  \begin{eqnarray*}
    - H_{\partial \Sigma, \Sigma} - [\rho A (\vartheta, \vartheta) - \rho H] &
    = & \sum \langle \nabla_{e_i} (\mu - \mu'), e_i \rangle .
  \end{eqnarray*}
  We note that $\langle \mu, \mu' \rangle = 1 + O (\rho^{- \tau})$, this
  implies that $\mu' : = a \vartheta + b \mu$ with $a, b - 1 = O (\rho^{-
  \tau})$. Then
  \begin{eqnarray*}
    \sum_i \langle \nabla_{e_i} (\mu - \mu'), e_i \rangle & = & \sum a \langle
    \nabla_{e_i} \vartheta, e_i \rangle + (b - 1) \langle \nabla_{e_i} \mu,
    e_i \rangle\\
    & = & \sum \langle \vartheta, \mu' \rangle H_{\partial \Sigma, \partial
    M} + O (\rho^{- 1 - 2 \tau})\\
    & = & O (\rho^{- 1 - \tau}) .
  \end{eqnarray*}
  So
  \[ (A - H g) (X, \vartheta) = - \rho H_{\partial \Sigma, \Sigma} - \rho
     \langle \vartheta, \mu' \rangle H_{\partial \Sigma, \partial M} + O
     (\rho^{- 1 - 2 \tau}) . \]
  Then the theorem is easily obtained by integration and discarding small
  terms.
\end{proof}

\begin{remark}
  Take $n = 3$ and that $\Sigma$ meets $\partial M$ orthogonally. Using as
  well Gauss-Bonnet theorem for surfaces with boundary, we have that
  {\eqref{eq:hawking}} turns into
  \[ \frac{1}{8 \pi} (\frac{2| \Sigma |}{4 \pi})^{1 / 2} (4 \pi \chi (\Sigma)
     - \frac{1}{2} \int_{\Sigma} H_{\Sigma, M}^2 \mathrm{d} \sigma) \]
  which reduces to
  \[ (\frac{| \Sigma |}{8 \pi})^{1 / 2} (1 - \frac{1}{8 \pi} \int_{\Sigma}
     H_{\Sigma, M}^2 \mathrm{d} \sigma) \]
  when $\Sigma$ is a standard hemisphere and $\chi (\Sigma) \equiv 2 - 2 g - b
  = 1$.
\end{remark}

\section{Isoperimetric mass}\label{isoperimetric derivation}

{\bfseries{Notations}} We introduce two more notations used in this section.
\begin{eqnarray}
  &  & \sigma_{i j} = g_{i j} - \delta_{i j} \nonumber\\
  \langle \cdot, \cdot \rangle_e &  & \text{ inner product in Euclidean
  metric} . \nonumber
\end{eqnarray}
We adapt it to our setting that

\begin{theorem}
  $m_{\ensuremath{\operatorname{ISO}}} = m_{\ensuremath{\operatorname{ADM}}}$.
\end{theorem}

\begin{proof}
  {\bfseries{Step 1, derivative of $\mathcal{A}$.}}
  
  Recall that $\mathrm{d} \sigma = (1 + h^{i j} \sigma_{i j} + O (r^{- 2
  \tau}))^{\frac{1}{2}} d \bar{\sigma}$ (see
  {\cite[(2.7)]{fan-large-sphere-2009}}), then
  \begin{equation}
    \mathcal{A} (r) = 2 \pi r^2 + \frac{1}{2} \int_{\Sigma} h^{i j} \sigma_{i
    j} \mathrm{d} \bar{\sigma} + O (r^{2 - 2 \tau})
  \end{equation}
  and
  \begin{align}
    \mathcal{A}' (r) & = 4 \pi r + \frac{1}{2} \int_{\Sigma}
    \frac{\partial}{\partial r} (h^{i j} \sigma_{i j}) \mathrm{d} \bar{\sigma}
    + \frac{1}{r} \int_{\Sigma} h^{i j} \sigma_{i j} \mathrm{d} \bar{\sigma} +
    O (r^{1 - 2 \tau})\\
    & = 4 \pi r + \frac{1}{2} \int_{\Sigma} h^{i j} \sigma_{i j, k}
    \frac{x^k}{r} \mathrm{d} \bar{\sigma} + \frac{1}{r} \int_{\Sigma} h^{i j}
    \sigma_{i j} \mathrm{d} \bar{\sigma} + O (r^{1 - 2 \tau})\\
    & = 4 \pi r + \frac{1}{2} \int_{\Sigma} \frac{\sigma_{i i, k} x^k}{r}
    \mathrm{d} \bar{\sigma} - \frac{1}{2} \int_{\Sigma} \frac{\sigma_{i j, k}
    x^i x^j x^k}{r^3} \mathrm{d} \bar{\sigma} + \frac{1}{r} \int_{\Sigma} h^{i
    j} \sigma_{i j} \mathrm{d} \bar{\sigma} + O (r^{1 - 2 \tau}) . \label{area
    derivative}
  \end{align}
  Note that
  \begin{align}
    \int_{\Sigma} \frac{\sigma_{i j, k} x^i x^j x^k}{r^3} \mathrm{d}
    \bar{\sigma} & = \int_{\Sigma} \frac{\partial}{\partial x^k}
    (\frac{\sigma_{i j} x^j}{r}) \frac{x^i x^k}{r^2} \mathrm{d} \bar{\sigma}\\
    & = - \int_{\Sigma} (\delta_{i k} - \frac{x^i x^k}{r^2})
    \frac{\partial}{\partial x^k} (\frac{\sigma_{i j} x^j}{r}) \mathrm{d}
    \bar{\sigma} + \int_{\Sigma} \frac{\partial}{\partial x^i}
    (\frac{\sigma_{i j} x^j}{r}) \mathrm{d} \bar{\sigma}.
  \end{align}
  Since $\frac{x^i}{r} \partial_k$ is normal to $\Sigma_r$ under Euclidean
  metric, $\delta_{i k} - \frac{x^i x^k}{r^2}$ is the Euclidean metric
  projected to $S_r$ (or in other words, induced metric), hence letting $Y_j =
  \sigma_{k j} \frac{x^k}{r}$ and $\bar{e}_1 = - \frac{\partial}{\partial x^1}
  $ we see that by first variation formula under Euclidean metric,
  \begin{align}
    & \int_{\Sigma} (\delta_{i k} - \frac{x^i x^k}{r^2})
    \frac{\partial}{\partial x^k} (\frac{\sigma_{i j} x^j}{r}) \mathrm{d}
    \bar{\sigma}\\
    = & \int_{\Sigma} \ensuremath{\operatorname{div}}^e_{\Sigma} Y\\
    = & \int_{\Sigma} H_e  \langle Y, \bar{n} \rangle_e \mathrm{d}
    \bar{\sigma} + \int_{\partial \Sigma} \langle \bar{e}_1, Y \rangle_e
    \mathrm{d} \bar{\theta}\\
    = & \int_{\Sigma} \frac{2}{r} \sigma_{i j} \frac{x^i}{r}  \frac{x^j}{r}
    \mathrm{d} \bar{\sigma} - \int_{\partial \Sigma} \sigma_{1 k}
    \frac{x^k}{r} \mathrm{d} \bar{\theta}\\
    = & \int_{\Sigma} \frac{2 \sigma_{i j} x^i x^j}{r^3} \mathrm{d}
    \bar{\sigma} - \int_{\partial \Sigma} g_{1 a} \frac{x^a}{r} \mathrm{d}
    \bar{\theta} .
  \end{align}
  Hence
  \begin{align}
    & \int_{\Sigma} \frac{\sigma_{i j, k} x^i x^j x^k}{r^3} \mathrm{d}
    \bar{\sigma}\\
    = & - 2 \int_{\Sigma} \frac{\sigma_{i j} x^i x^j}{r^3} \mathrm{d}
    \bar{\sigma} - \int_{\partial \Sigma} g_{1 a} \frac{x^a}{r} \mathrm{d}
    \bar{\theta}\\
    & \hspace{1em} + \int_{\Sigma} \frac{\sigma_{i j, i} x^j}{r} \mathrm{d}
    \bar{\sigma} + \int_{\Sigma} \sigma_{i j} (\frac{\delta_{i j}}{r} -
    \frac{x^i x^j}{r^3}) \mathrm{d} \bar{\sigma}\\
    = & - 2 \int_{\Sigma} \frac{\sigma_{i j} x^i x^j}{r^3} \mathrm{d}
    \bar{\sigma} + \int_{\partial \Sigma} \sigma_{1 a} \frac{x^a}{r}
    \mathrm{d} \bar{\theta}\\
    & \hspace{1em} + \int_{\Sigma} \frac{\sigma_{i j, i} x^j}{r} \mathrm{d}
    \bar{\sigma} + \frac{1}{r} \int_{\Sigma} h^{i j} \sigma_{i j} \mathrm{d}
    \bar{\sigma} + O (r^{1 - 2 \tau}) .
  \end{align}
  Inserting this back to {\eqref{area derivative}}, we obtain
  \begin{align}
    \mathcal{A}' (r) & = 4 \pi r - \frac{1}{2} \left[ \int_{S_r}
    (\frac{\sigma_{i j, i} x^j}{r} - \frac{\sigma_{i i, k} x^k}{r}) \mathrm{d}
    \bar{\sigma} + \int_{\partial S_r} g_{1 a} \frac{x^a}{r} \mathrm{d}
    \bar{\theta} \right]\\
    & \hspace{5em} + \frac{1}{2 r} \int_{\Sigma} h^{i j} \sigma_{i j}
    \mathrm{d} \bar{\sigma} + \int_{\Sigma} \frac{\sigma_{i j} x^i x^j}{r^3}
    \mathrm{d} \bar{\sigma} + O (r^{1 - 2 \tau})\\
    & = 4 \pi r - 4 \pi m + \frac{1}{2 r} \int_{\Sigma} h^{i j} \sigma_{i j}
    \mathrm{d} \bar{\sigma} + \int_{\Sigma} \frac{\sigma_{i j} x^i x^j}{r^3}
    \mathrm{d} \bar{\sigma} + o (1)\\
    & = 2 \pi r - 4 \pi m + \frac{\mathcal{A} (r)}{r} + \frac{1}{r}
    \int_{\Sigma} \frac{\sigma_{i j} x^i x^j}{r^2} \mathrm{d} \bar{\sigma} + o
    (1), \label{area derivative by mass}
  \end{align}
  where in the last line we used {\eqref{area derivative}}.
  {\bfseries{Step 2, asymptotics of volume $V (r)$.}}
  
  Recall {\cite[(2.2)]{fan-large-sphere-2009}},
  \begin{equation}
    | \nabla r|^2 = 1 - \frac{\sigma_{i j} x^i x^j}{r^2} + O (r^{- 2 \tau}) .
  \end{equation}
  By co-area formula,
  \[ V' (r) = \int_{\Sigma} \frac{1}{| \nabla r|} d \sigma =\mathcal{A} (r) +
     \frac{1}{2} \int_{\Sigma} \frac{\sigma_{i j} x^i x^j}{r^2} d \sigma + O
     (r^{2 - 2 \tau}) . \]
  Again by $\mathrm{d} \sigma = (1 + h^{i j} \sigma_{i j} + O (r^{- 2
  \tau}))^{\frac{1}{2}} \mathrm{d} \bar{\sigma}$, we have that
  \begin{equation}
    V' (r) = \int_{\Sigma} \frac{1}{| \nabla r|} \mathrm{d} \sigma
    =\mathcal{A} (r) + \frac{1}{2} \int_{\Sigma} \frac{\sigma_{i j} x^i
    x^j}{r^2} \mathrm{d} \bar{\sigma} + O (r^{2 - 2 \tau}) .
  \end{equation}
  Note that the derivative $\mathcal{A}' (r)$ in {\eqref{area derivative by
  mass}}, we have that
  \[ \mathcal{A}' (r) = 2 \pi r - 4 \pi m + \frac{\mathcal{A} (r)}{r} +
     \frac{2}{r} (V' (r) -\mathcal{A}(r)) + o (1) \]
  which gives
  \[ (r\mathcal{A}(r))' = 2 \pi r^2 - 4 \pi m r + 2 V' (r) + o (r) . \]
  Integration then gives
  \begin{equation}
    V (r) = \frac{1}{2} r\mathcal{A} (r) - \frac{1}{3} \pi r^3 + \pi m r^2 + o
    (r^2) .
  \end{equation}
  {\bfseries{Step 3, evaluation of isoperimetric mass.}}
  \begin{align}
    & \frac{2}{\mathcal{A} (r)} (V (r) - \frac{\sqrt{2} \mathcal{A}^{3 / 2}
    (r)}{6 \sqrt{\pi}})\\
    = & \frac{2}{\mathcal{A} (r)} [\frac{1}{2} r\mathcal{A}(r) - \frac{1}{3}
    \pi r^3 + \pi m r^2] - \frac{\sqrt{2} \mathcal{A}^{\frac{1}{2}} (r)}{3
    \sqrt{\pi}} + o (1)\\
    = & r + \frac{2 \pi r^2}{\mathcal{A} (r)} (m - \frac{r}{3}) - \frac{2
    r}{3} (\frac{\mathcal{A}(r)}{2 \pi r^2})^{\frac{1}{2}} + o (1)\\
    = & r + (m - \frac{r}{3}) (1 - \Theta + O (r^{- 2 \tau})) - \frac{2 r}{3}
    (1 + \frac{1}{2} \Theta + O (r^{- 2 \tau})) + o (1)\\
    = & m + o (1),
  \end{align}
  where we have used that
  \[ \Theta := \frac{1}{4 \pi r^2} \int_{\Sigma} h^{i j} \sigma_{i j}
     \mathrm{d} \bar{\sigma} = O (r^{- \tau}) . \]
  so that
  \[ \mathcal{A} (r) = 2 \pi r^2 (1 + \Theta + O (r^{- 2 \tau})) . \]
\end{proof}

\begin{remark}
  Because we can also use arbitrary sets of finite perimeter $\Omega_i$ to
  define the isoperimetric mass, i.e.
  \begin{equation}
    \tilde{m}_{\ensuremath{\operatorname{ISO}}} = \limsup_{\Omega_i \to M}
    \frac{2}{| \partial^{\ast} \Omega_i \cap \partial M|} (\mathcal{H}^3
    (\Omega_i) - \frac{\sqrt{2} | \partial^{\ast} \Omega_i \cap \partial M|}{6
    \sqrt{\pi}}) \label{isoperimetric mass 2}
  \end{equation}
  where $\partial^{\ast} \Omega$ is the reduced boundary of $\Omega_i$.We see
  that $\tilde{m}_{\ensuremath{\operatorname{ISO}}} \geqslant
  m_{\ensuremath{\operatorname{ISO}}} = m_{\ensuremath{\operatorname{ADM}}}$.
\end{remark}

\end{document}